\documentclass[draft]{amsart}
\usepackage{enumitem}
\usepackage{amssymb,amsmath}

\parskip=\smallskipamount

\newtheorem{theorem}{Theorem}
\newtheorem{lemma}[theorem]{Lemma}

\newtheorem{proposition}{Proposition}

\theoremstyle{definition}

\newtheorem{remark}[theorem]{Remark}

\newcommand{\C}{\mathbb{C}}

\newcommand{\Z}{\mathbb{Z}}

\newcommand{\R}{\mathbb{R}}

\newcommand{\Grr}{{\mathrm{Gr}}_\R}
\newcommand{\Grc}{{\mathrm{Gr}}_\C}
\newcommand{\Gl}{{\mathrm{Gl}}}

\newcommand{\cC}{\mathcal{C}}

\newcommand{\cM}{\mathcal{M}}

\newcommand{\del}{\partial}
\newcommand{\delb}{\bar\partial}

\newcommand\hra{\hookrightarrow}

\newcommand\Real{\mathrm{Re}}

\newcommand\grad{\mathrm{grad}}
\newcommand\dist{\mathrm{dist}}
\newcommand\diag{\mathrm{diag}}

\def\bar{\overline}

\numberwithin{equation}{section}

\begin{document}
\title[Isotopy of codimension 2 real submanifolds]
{On complex points of codimension 2 submanifolds}
\author{Marko Slapar}
\address{University of Ljubljana, Faculty of Education, Kardeljeva Plo\v s\v cad 16, 1000 Ljubljana, Slovenia and Institute of Mathematics, Physics and Mechanics, jadranska 19, 1000 Ljubljana, Slovenia}
\email{marko.slapar@pef.uni-lj.si}
\thanks{Supported by the research program P1-0291 and research project J1-5432 at the Slovenian Research Agency.}

%
%
\date{\today} 

\begin{abstract}
In this paper we study the structure of complex points of codimension $2$ real submanifolds in complex $n$ dimensional manifolds. We show that the local structure of a complex point up to isotopy only depends on their type (either elliptic or hyperbolic). We also show that any such submanifold can be smoothly isotoped into a submanifold that has $2$-strictly pseudoconvex neighborhood basis.   
\end{abstract}
\maketitle

\section{Introduction} 
Let $i\!:Y\hra X$ be a  real compact $2n$ dimensional manifold $Y$, smoothly embedded into an $n+1$ dimensional complex manifold $(X,J)$. The expected dimension of the complex tangent space $T^\C_pY=T_pY\cap JT_pY$  equals to $\max\{0,n-1\}$. Such points are called {\it{CR regular}}, and their complement in $Y$ consists of {\it{CR singular}} points, i.e., points with $\dim_\C T^\C_pY=n$. Thus CR singular points are exactly those points, for which $T_pY\subset T_pX$ is a complex subspace, and we rather call them {\it{complex points}}. By Thom transversality theorem \cite{Thom}, complex points are isolated  for generic embeddings.  If $Y$ is also oriented, we call a complex point $p$ {\it positive}, if the orientation of $T_pY$ agrees with the induced orientation of $T_pY$ as a complex subspace of $T_pX$, and {\it negative}, if the orientation is opposite. 

Using local coordinates on $Y$ and $X$ near a complex point $p$, we can assume that locally $i\!:V\hra U$, where $V$ is a neighborhood of $0=p$ in $\R^{2n}=\C^{n}$ and $U$ a neighborhood of $0=i(p)$ in $\C^{n+1}$. By the Gauss map $Di\!:y\mapsto T_y\in \Grr(2n,\C^{n+1})$, the complex point $0$ is mapped into $\Grc(n,\C^{n+1})\subset \Grr(2n,\C^{n+1})$. Genericity of $i$ means that the intersection of $Di(V)$ with $\Grc(n,\C^{n+1})$ is transverse. Note that the sign of the intersection does not depend on the orientation of $V$, and we call complex point $p$ {\it elliptic}, if the sign of the intersection is positive and {\it hyperbolic}, if the sign of the intersection is negative. We use this terms more freely than defined by Dolbeault, Tomassini and Zaitsev in \cite{DTZ1,DTZ2}, where the term elliptic is reserved for more special classes of complex points. We denote by $e(Y)$ the number of elliptic complex points and by $h(Y)$ the number of hyperbolic complex points (and $e_\pm(Y)$, $h_\pm(Y)$ in the case when $Y$ is oriented, where we can distinguish between positive and negative complex points). The algebraic counts $I(Y)=e(Y)-h(Y)$ and $I_{\pm}(Y)=e_\pm(Y)-h_\pm(Y)$, called \textit{Lai indices}, are isotopy invariant and can be expressed by formulas \cite{Lai}:

\[I=\chi(Y)+\sum_{k=1}^{\lfloor(n+1)/2\rfloor-1} <e(\nu Y)^{2k+1}\cup c_{n-2k-1}(X)|_Y,[Y]>\]

and in the oriented case

\[2I_\pm=\chi(Y)+\sum_{k=0}^{n}(\pm 1)^{k+1} <e(\nu Y)^{k}\cup c_{n-r}(X)|_Y,[Y]>,\]

\noindent where $\chi(Y)$ is the Euler characteristic of $Y$, $e(\nu Y)\in H^2(Y,\Z)$ is the Euler class of the normal bundle $\nu Y$ of $Y$ and  $c_j(X)\in H^{2j}(X,\Z)$ the $j-$th Chern class of $X$. 

The structure of complex points is well understood and already classical in the case $n=1$ (real surfaces in complex surfaces). Local theory has been first studied by Bishop \cite{Bis} and later mostly by Moser and Webster \cite{MoW}, where suitable normal forms have been found. Global theory has been studied by Eliashberg and Harlamov \cite{EH} and Forstneri\v c \cite{F}, where is was shown that a pair of elliptic and hyperbolic complex point can always be canceled by a $\cC^0$ small isotopy (the two complex points must be of the same sign, if the surface is oriented). From these works it also follows that the presence of elliptic complex points is the only obstruction to having tubular Stein neighborhood basis of the surface; see also \cite{S1}. The local structure up to quadratic part in dimension $n=2$ has been studied by Coffman \cite{C1}, and up to isotopy by the author in \cite{S2}. For $n>1$, some normal forms have been obtained recently by Yin and Huang in \cite{YH1,YH2} and Burcea \cite{Bur}. The analogous cancellation theorem to the one for $n=1$ was proved in \cite{S3}. In the case of $n>1$ one cannot get Stein neighborhoods for obvious topological reasons, but $2$-strictly pseudoconvex  neighborhoods after isotopy of the manifold $Y$ have been constructed by the author in \cite{S2} in the case $n=2$. Some results on Bishop discs \cite{Bis}, have been extended to higher dimensions in \cite{DTZ1,DTZ2}. 

\begin{theorem}\label{thm1} Every smooth embedding of a $2n$ dimensional compact real manifold $Y$ into a $n+1$ dimensional complex manifold $X$ can be deformed by a smooth isotopy to a manifold only having complex point of the following two types:
\begin{itemize}
\item[i)] $w=|z_1|^2+|z_2|^2+\cdots +|z_n|^2$,
\item[ii)] $w=\overline{z_1}^2+|z_2|^2+\cdots +|z_n|^2$.
\end{itemize}
The isotopy is $\cC^0$ close to the original embedding.
\end{theorem}
 
\begin{remark} In the above theorem elliptic complex points are deformed to be of type i) and hyperbolic to be of type ii).  We can take any other suitable model situations for complex points. For example $w=\frac{1}{2}\bar z_1^2+ \frac{1}{2}\bar z_2^2+\cdots + \frac{1}{2}\bar z_n^2$ and $w=|z_1|^2+ \frac{1}{2}\bar z_2^2+\cdots + \frac{1}{2}\bar z_n^2$. The first one is elliptic for even $n$ and hyperbolic for odd $n$, and the second one is elliptic for odd $n$ and hyperbolic for even $n$.  
\end{remark}

Theorem \ref{thm1}, together with the cancellation theorem proved in \cite{S3}, shows that Lai indices are the only topological invariants of complex points up to isotopy. 

We call a $\cC^2$ function $\phi$ \textit{$q$-pseudoconvex}, if the Levi form $\frac{i}{2}\del\delb \phi$ has at least $q$ nonnegative eigenvalues, and \textit{$q$-strictly pseudoconvex} if at least $q$ eigenvalues are strictly positive. Level sets of $q$-strictly pseudoconvex functions are called \textit{$q-1$-strictly pseudoconvex} and we call a domain \textit{$q$-strictly pseudoconvex}, if it has $q-1$-strictly pseudoconvex boundary.    

\begin{theorem}\label{thm2} Let $Y\hookrightarrow X$ be a smooth embedding of a compact $2n$-manifold in a complex $n+1$ dimensional manifold. After a $\cC^0$ small smooth isotopy, $Y$ has a $2$-strictly pseudoconvex tubular neighborhood basis.  
\end{theorem}

In the above theorem, we call a neighborhood tubular, if it is diffeomorphic to the normal bundle of $Y$. The above results extend the results from \cite{S2}.

\section{On normal forms of complex points up to quadratic term}

Using appropriate local coordinates, we can assume that near an isolated complex point $p\in Y$, the manifold $Y$ is given by the graph of a smooth function $f:\C^n\to \C$, where $p=0$ and the complex tangent space at $0$ equals $\C^n$. Using Taylor expansion of $f$, we can arrange $Y$ to be locally of the form
\begin{equation}\label{form} w=\bar z^TAz+\Real (z^TBz)+o(|z|^2),\end{equation}
where $z=(z_1,z_2,\ldots,z_n)$, and $A,B$ are $n\times n$ complex matrices. The matrix $B$ can be assumed to be symmetric. Calculating the intersection index, \cite{C1}, one can see that the point $p$ is elliptic, if the determinant of
\begin{equation}\label{det}\left[
\begin{array}{cc}
A & \bar B \\
B & \bar A
\end{array}\right]
\end{equation}
is positive, and hyperbolic, if it is negative. For this reason, we will call a pair of matrices $(A,B)$, $B=B^T$, elliptic, if the above determinant is positive, and hyperbolic, if it is negative. Using a linear change of coordinates of the form
\[\begin{pmatrix}
z \\ 
w 
\end{pmatrix}
=
\begin{pmatrix}
P & *\\
0 & c
\end{pmatrix}
\begin{pmatrix}
\widetilde z\\
\widetilde w
\end{pmatrix},
\]
where $P$ is some nonsingular $n\times n$ matrix, the form ($\ref{form}$) is transformed into

$$w=\frac{1}{c}\bar z^TP^*APz+\frac{1}{2}\left(\frac{1}{c}z^TP^TBPz+\frac{1}{c}{\bar z}^T{\bar P}^T\bar B\bar P\bar z\right)+o(|z|^2),$$

\noindent where we have dropped the $\sim$ in the new coordinates. After a quadratic holomorphic change in $w$, the equation becomes

$$w=\frac{1}{c}\bar z^TP^*APz+\frac{1}{\overline{c}}\Real (z^TP^TBPz)+o(|z|^2).$$

We see the group $G=S^1\times \Gl(n,\C)/\Z_2$ (the $\Z_2$ quotient means $P\sim-P$) acts on the space of pairs of matrices 
\[\cM_n=\{(A,B);\ A,B\in M_{n}(\C),\ B=B^T\},\] 
that determine the quadratic part in (\ref{form}), by
\[(\zeta,P)(A,B)=(\zeta P^*AP,\bar \zeta P^TAP).\]  
Any other biholomorphic change of coordinates, preserving the complex point and the tangent space at that point, has the same effect on the quadratic part as an above linear change, since the transformation up to quadratic term is determined by its linearization at $0$. The moduli space of quadratic terms of (\ref{form}) up to biholomorphic change of coordinates is thus equal to the quotient space $\cM_n/G$. We call two pairs $(A_0,B_0),\ (A_1, B_1)\in M$ {\it G-congruent}, if they are in the same orbit of this group action. Let us denote by $\cM_n^+$ the set of elliptic pairs from $\cM_n$, and by $\cM_n^-$ the part corresponding to hyperbolic pairs. The corresponding quotients $\cM_n^\pm/G$ are the moduli spaces of elliptic and hyperbolic complex points up to their quadratic term.  

It is not easy to find nice canonical forms for pairs $(A,B)\in \cM_n$ up to $G$-congruence. One can first put $A$ in a canonical form under $^*$congruence using a result of Horn and Sergeichuk \cite[Theorem 1]{HS}, and then use the $^T$congruence with matrices stabilizing that canonical form to simplify $B$. The complete description in dimension $n=2$ can be found in \cite[Theorem 7.2]{C1}, but the process seems to be too computational to nicely adapt to higher dimensions. The other approach is the following. By Tagaki factorization \cite[Corollary 4.4.4]{HJ}, any complex symmetric matrix is $^T$congruent with a unitary matrix to a diagonal matrix with nonnegative entries. Assuming that the matrix $B$ is nonsingular, we can thus reduce $B$ to the identity matrix by a $^T$congruence. One can than try to simplify $A$ by finding a canonical form for $^*$congruence by complex orthogonal matrices, i.e. with the property $VV^T=I$. If for example $A>0$, one can put (\ref{form}) by a biholomorphic change of coordinates into a form
\[w=|z|^2+\sum_{k=1}^{n}\frac{\gamma_k}{2}\left(z^2+\bar z^2\right)+o(|z|^2),\]
where $0\le\gamma_1\le\gamma_2\le\cdots\le\gamma_{n}$ are called generalized Bishop invariants. This might be well known, but we have not seen it mentioned in print. If $A$ be positive semidefinite and $B|_V:V\to V$ is nondegenerate, where $V$ is the direct sum of eigenspaces of $A$, corresponding to strictly positive eigenvalues of $A$, then $B$ can also be made diagonal (see \cite[Theorem 4.5.15]{HJ}). So in this case, one can put (\ref{form}) by a biholomorphic change of coordinates into a form
\[w=\sum_{k=1}^{j}|z_k|^2+\cdots+\sum_{k=1}^{n}\frac{\gamma_k}{2}\left(z^2+\bar z^2\right)+o(|z|^2),\]
where $0\le\gamma_1\le\gamma_2\le\cdots\le\gamma_{k}$ and $\gamma_{j+1},\ldots,\gamma_{n}$ are either $0$ or $1$. Example $A=\left[\begin{smallmatrix} 1&0\\0&0\end{smallmatrix}\right]$, $B=\left[\begin{smallmatrix} 0&b\\ b&0\end{smallmatrix}\right]$, shows that such nice form cannot be achieved without some extra assumption on $B$. The situation is even more complicated for general Hermitian $A$ (note that a complex point can be made quadratically flat, i.e. quadratic part is real, if and only if $A$ is Hermitian up to multiplication by a complex number). 
\begin{remark}\label{split} We can easily see, by changing the order of rows and column of the matrix in (\ref{det}), that if 
$A=A_1\oplus A_2$ and $B=B_1\oplus B_2$, where $A_1,B_1$ are $n'\times n'$ matrices and $A_2,B_2$ are $(n-n')\times (n-n')$ matrices, that 
\[\det\left[
\begin{array}{cc}
A & \bar B \\
B & \bar A
\end{array}\right]=\det\left[
\begin{array}{cc}
A_1 & \bar B_1 \\
B_1 & \bar A_1
\end{array}\right]\det\left[
\begin{array}{cc}
A_2 & \bar B_2 \\
B_2 & \bar A_2
\end{array}\right].
\]
If a pair of matrices $(A,B)\in\cM_n$ has such splitting, we will write $(A,B)=(A_1,B_1)\oplus(A_2,B_2)$. If, for example, $(A_1,B_1)\in\cM_{n'}^+$, the ellipticity or hyperbolicity of $(A,B)$ is determined by that of $(A_2,B_2)$.    
\end{remark}
\section{Consimilarity of matrices}
Two matrices $A,B\in M_n$ are called {\it consimilar}, if there exists a matrix $S\in \Gl(n,\C)$, so that $B=SA\bar S^{-1}$. It turns out that two nondegenerate matrices $A$ and $B$ are consimilar if and only if the matrices $A\bar A$ and $B\bar B$ are similar (the statement is slightly more complicated for singular matrices)\cite[Theorem 4.1]{HH}. For any matrix $A$, the nonreal eigenvalues of $A\bar A$ come in conjugate pairs $c, \bar c$ (the characteristic polynomial has real coefficients), and the negative eigenvalues all have even algebraic multiplicity (\cite[Section 4.6]{HJ}). 

Let us now explain the canonical form under consimilarity of a nonsingular matrix $A$; the proof can be found in \cite{HH}. We will use the notation $J_k(\lambda)$ for the $k\times k$ Jordan block with $\lambda$ on the diagonal. Let
\[J_{pos}(A\bar A)\oplus J_{neg}(A\bar A)\oplus J_{com}(A\bar A)\]
be the Jordan decomposition of $A\bar A$, where
\begin{itemize}
\item $J_{pos}(A\bar A)=J_{k_1}(\mu_1)\oplus\cdots\oplus J_{k_r}(\mu_r)$,
\item $J_{neg}(A\bar A)=\left[J_{l_1}(\nu_1)\oplus J_{l_1}(\nu_1)\right]\oplus\cdots\oplus \left[J_{l_s}(\mu_s)\oplus J_{l_s}(\mu_s)\right]$, 
\item $J_{com}(A\bar A)=\left[J_{m_1}(c_1)\oplus J_{m_1}(\bar{c_1})\right]\oplus\cdots\oplus \left[J_{m_u}(c_u)\oplus J_{m_u}(\bar{c_u})\right]$,
\end{itemize}
where $\mu_i$ are positive eigenvalues, $\nu_i$ negative eigenvalues and $c_i,\bar{c_i}$ conjugate pairs of nonreal eigenvalues of $A\bar A$. Then the canonical form of $A$ under consimilarity is given (uniquely up to permutation of direct summands) by 
\begin{equation}\label{canonical}J_{P}(A)\oplus Q_{N}(A)\oplus Q_{C}(A)\end{equation}
where  
\begin{itemize}
\item $J_{P}(A)=J_{k_1}(\sqrt{\mu_1})\oplus\cdots\oplus J_{k_r}(\sqrt{\mu_r})$,
\item $J_{N}(A)=\left[\begin{smallmatrix}0&I_{l_1}\\J_{l_1}(\nu_1)&0\end{smallmatrix}\right]\oplus\cdots\oplus \left[\begin{smallmatrix}0&I_{l_s}\\J_{l_s}(\nu_s)&0\end{smallmatrix}\right]$, 
\item $J_{C}(A)=\left[\begin{smallmatrix}0&I_{m_1}\\J_{m_1}(c_1)&0\end{smallmatrix}\right]\oplus\cdots\oplus \left[\begin{smallmatrix}0&I_{m_u}\\J_{m_u}(c_u)&0\end{smallmatrix}\right]$.
\end{itemize}
By $I_k$ we denote the $k\times k$ identity matrix.
 
\begin{proposition}\label{prop} Let $A$ be a nonsingular matrix. After a small (generic) perturbation $\widetilde A$ of $A$, there exists a nonsingular complex matrix $S$, so that
\begin{equation}\label{diagonal}S\widetilde{A}\bar S^{-1}=\left[\begin{matrix} D&0\\ 0&\Lambda\end{matrix}\right],\end{equation}
where $D$ is diagonal with strictly positive entries, and $\Lambda$ is a block diagonal matrix with $2\times 2$ blocks of the form $\left[
\begin{smallmatrix}
0 & 1 \\                                                                                     
\lambda_j & 0
\end{smallmatrix}\right]$, with $\lambda_j$ all either nonreal or strictly negative (all can be chosen do be distinct), and the number of each type (strictly positive, strictly negative or nonreal) entries in the canonical forms of $A$ and $\widetilde A$ under consimilarity is the same.
\end{proposition}
\begin{proof}
Let $T$ be such that $TA\bar T^{-1}$ is of the form (\ref{canonical}). Let $J_k(\mu)$ be one of the Jordan blocks in the decomposition of $J_P$. Let $E=\diag(\epsilon_1,\epsilon_2,\ldots,\epsilon_k)$ be the $k\times k$ diagonal matrix with $\epsilon_1,\epsilon_2,\ldots,\epsilon_k$ on the diagonal, where $\epsilon_i$ are all small and real, and $\epsilon_i\ne\epsilon_j$ if $i\ne j$. Let  
$J_k^E(\mu)=J_k(\mu)+E$.
Then 
\begin{equation}\label{e1} J_k^E(\mu)\overline{J_k^E(\mu)}=\diag\left((\mu+\epsilon_1)^2,(\mu+\epsilon_2)^2,\ldots,(\mu+\epsilon_k)^2\right)+U
,\end{equation}
where $U$ is some strictly upper triangular matrix. Denote by $\widetilde{J_P(A)}$ the perturbation of the matrix $J_P(A)$ by doing all such perturbations on its direct summands that are not $1 \times 1$ matrices.  
On the blocks $\left[\begin{smallmatrix}0&I_{m}\\J_{m}(b)&0\end{smallmatrix}\right]$ from $Q_N(A)$ or $Q_C(A)$, we do a similar perturbation. Let $\Delta=\diag(\delta_1,\ldots,\delta_m)$, where $\delta_i$ are small, all distinct and positive. Then 
\begin{equation}\label{e2}\left[\begin{smallmatrix}0&I_m\\J_m(b)+\Delta&0\end{smallmatrix}\right] 
\overline{\left[\begin{smallmatrix}0&I_m \\
J_m(b)+\Delta&0\end{smallmatrix}\right]}
=\diag\left(b+\delta_1,\ldots,b+\delta_m,\bar b+\delta_1,\ldots,\bar b+\delta_n\right)+V, \end{equation}
where $V$ is some strictly upper triangular matrix. Let $\widetilde{J_N(A)}$ and $\widetilde{J_C(A)}$ be the results of doing such perturbation on all the direct summand of $Q_N(A)$ and $Q_C(A)$ that are not $2\times 2$ matrices. Let
\[\widetilde{A}=T^{-1}(\widetilde{J_P(A)}\oplus \widetilde{Q_N(A)}\oplus \widetilde{Q_C(A)})\bar T.\]
We see from (\ref{e1}) and (\ref{e2}) that all positive and nonreal eigenvalues of $\widetilde{A}\overline{\widetilde{A}}$ are distinct and have geometric multiplicity $1$, while all negative eigenvalues have both algebraic and geometric multiplicity equal to $2$. The nonreal eigenvalues appear in conjugate pairs. If the perturbations are small enough, the number of each of the type (positive, negative or nonreal) of eigenvalues, counted with multiplicity, is the same as that of $A\bar A$, so the canonical form of $\widetilde A$ under consimilarity is of the desired form.    
\end{proof}            

\section{Proof of main theorems}
Let us first introduce some notations that we use in the proofs. A homotopy of pairs $(A_t,B_t)\in M_n$ is a smooth map from the interval $[0,1]$ into $\cM_n$, meaning that $A_t,B_t$ are $n\times n$ matrices, and $B_t=B_t^T$. We call such homotopy nondegenerate, if the determinant $\left|\begin{smallmatrix} A_t &\bar B_t\\ B_t&\bar A_t\end{smallmatrix}\right|$ is never $0$. Such a homotopy is completely contained in either $\cM_n^+$ or $M_n^-$ and thus preserves ellipticity or hyperbolicity. Since $G$ is path connected, any two $G$-congruent pairs $(A_0,B_0)$ and $(A_1,B_1)$ can be joined by a nondegenerate homotopy of pairs, since the sign of the determinant is preserved by $G$-congruency.  
\begin{proposition} Let $(A,B)\in \cM_n$. If $(A,B)$ is elliptic, then there exists a nondegenerate smooth homotopy from $(A,B)$ to the pair $(I_n,0_n)$. If $(A,B)$ is hyperbolic, then there exists a smooth nondegenerate homotopy from $(A,B)$ to $(0,1)\oplus (I_{n-1},0_{n-1})$.\label{connected}\end{proposition}  
\begin{proof}
Let $(A,B)\in \cM_n$ be a nondegenerate pair, meaning that $(A,B)$ is contained in either $\cM_n^+$ or $\cM_n^-$. By a small perturbation that can be achieved by nondegenerate homotopy, we can assume that both $A$ and $B$ are nonsingular. Since $B$ is symmetric, there exists a nonsingular matrix $S$, so that $S^TBS=I$ (\cite[Theorem 4.5.12]{HJ}). Let $A_1=S^*AS$. Since $(A,B)$ is $G$-congruent to $(A_1,I)$, the pairs can be joined by a nondegenerate homotopy. For convenience we drop the subscript and work with pairs $(A,I)$. Then
\[\left|\begin{matrix}A&I\\I&\bar A\end{matrix}\right| 
=(-1)^n\left|\begin{matrix}I&\bar A\\A& I\end{matrix}\right|=\det(A\bar A-I)=\prod_{k=1}^n(\lambda_k-1),\]
where $\lambda_k$ are eigenvalues of $A\bar A$. We use Proposition \ref{canonical} to get (perhaps after a small perturbation of A, that does not change the type of the pair (A,I)) a nonsingular $S$, so $SA\bar S^{-1}$ is of the form (\ref{diagonal}), where $D$ is a diagonal $n'\times n'$ matrix, and $\Lambda$ is block diagonal $(n-n')\times (n-n')$ matrix (note that $n-n'$ is even), with blocks of the form 
$\left[
\begin{smallmatrix}
0 & 1 \\                                                                                     
\lambda & 0
\end{smallmatrix}\right]$ and $\lambda$ either negative or nonreal.  Let $S_t$ be a homotopy from $I$ to $S$ in $\Gl(n,\C)$. Since $\left(S_t A\overline{S_t}^{-1}\right)\overline{\left(S_t A\  \overline{S_t}^{-1}\right)}=\left(S_t A\bar AS_t^{-1}\right)$, we have
\[\left|\begin{matrix}A&I\\I&\bar A\end{matrix}\right|=\left|\begin{matrix}S_tA\overline{S_t}^{-1} &I\\I&\overline{S_tA\overline{S_t}^{-1}}\end{matrix}\right|.\]
The homotopy of pairs $(S_tA\overline{S_t}^{-1},I)$ from  $(A,I)$ to $(D\oplus\Lambda,I_{n'}\oplus I_{n-n'})$ is therefore nondegenerate. We can assume that the first $l$ diagonal entries of $D=\diag(d_1,d_2,\ldots,d_{n'})$ are strictly smaller than $1$, and the other $n'-l$ are strictly grater than $1$. There are no entries equal to $1$, since the pair is nondegenerate. We can write the pair  $(D\oplus\Lambda,I_{n'}\oplus I_{n-n'})$ as a direct sum (see Remark \ref{split}) as either
\begin{align*}(D\oplus\Lambda,I_{n'}\oplus I_{n-n'})= 
(d_1,1)\bigoplus_{k=1}^{\lfloor l/2\rfloor}\left(\left[\begin{smallmatrix}d_{2k}&0\\0&d_{2k+1}\end{smallmatrix}\right],I_2\right) 
&\bigoplus_{k=l+1}^{n'}(d_k,1)\bigoplus_{k=1}^{\frac{n-n'}{2}}\left(\left[\begin{smallmatrix}0&1\\ b_k&0\end{smallmatrix}\right],I_2\right),
\end{align*}
if $l$ is odd, or
\[(D\oplus\Lambda,I_{n'}\oplus I_{n-n'})=\bigoplus_{k=1}^{l/2}\left(\left[\begin{smallmatrix}d_{2k-1}&0\\0&d_{2k}\end{smallmatrix}\right],I_2\right)\bigoplus_{k=l+1}^{n'}(d_k,1)\bigoplus_{k=1}^{\frac{n-n'}{2}}\left(\left[\begin{smallmatrix}0&1\\ b_k&0\end{smallmatrix}\right],I_2\right),\]
if $l$ is even. Here we have used $b_1,b_2,\ldots,b_{(n-n')/2}$ to denote all strictly negative or nonreal eigenvalues of $A\bar A$ appearing in the canonical decomposition of $A$ under consimilarity. We now show that in the case $l$ even, all the summands can be homotoped by nondegenerate homotopies of pairs into either $(1,0)\in\cM_1^+$ or $(I_2,0_2)\in\cM_2^+$, and in the case $n$ odd, this can be done for all but the first summand, which can be homotoped into the pair $(0,1)\in\cM_1^-.$  

Let $\left(\left[\begin{smallmatrix}0&1\\b&0\end{smallmatrix}\right],I_2\right)$ be one of the direct summands with $b$ either strictly negative or nonreal. Let $b(t)$ be a homotopy in $\C$ from $b$ to $-1$ that does not intersect $[0,\infty)$. We use the homotopy of pairs $\left(\left[\begin{smallmatrix}0&1\\b(t)&0\end{smallmatrix}\right],(1-t)I_2\right)$ to get to the pair $\left(\left[\begin{smallmatrix}0&1\\-1&0\end{smallmatrix}\right],0_2\right)$. Since
\[\det\left[\begin{smallmatrix}0&1&1-t&0\\b(t)&0&0&1-t\\1-t&0&0&1\\0&1-t&\overline{b(t)}&0\end{smallmatrix}\right]=|b_t-(1-t)^2|^2>0,\] this homotopy is nondegenerate. This pair can be homotoped by a nondegenerate homotopy $\left(\left[\begin{smallmatrix}0&1\\e^{(1-t)i\pi}&0\end{smallmatrix}\right],0_2\right)$ to $\left(\left[\begin{smallmatrix}0&1\\1&0\end{smallmatrix}\right],0_2\right)$. The $2\times 2$ matrix $\left[\begin{smallmatrix}0&1\\1&0\end{smallmatrix}\right]$ is $^*$congruent to $\left[\begin{smallmatrix}1&0\\0&-1\end{smallmatrix}\right]$ by a matrix $\frac{1}{\sqrt{2}}\left[\begin{smallmatrix}1&-1\\1&1\end{smallmatrix}\right]$. Let $H_t$ be a homotopy in $\Gl(2,\C)$ from $I_2$ to $\frac{1}{\sqrt{2}}\left[\begin{smallmatrix}1&-1\\1&1\end{smallmatrix}\right]$. Then the nondegenerate homotopy $\left(H_t^*\left[\begin{smallmatrix}0&1\\1&0\end{smallmatrix}\right]H,0_2\right)$ gives us the pair
$\left(\left[\begin{smallmatrix}1&0\\0&-1\end{smallmatrix}\right],0_2\right).$ We follow this by the nondegenerate homotopy $\left(\left[\begin{smallmatrix}1&0\\0&e^{(1-t)i\pi}\end{smallmatrix}\right],0_2\right)$ to get to $\left(I_2,0_2\right).$ 

The summands $(d_k,1)\in\cM_1^+$, with $k=l+1,\ldots,n'$ can easily be deformed to $(1,0)$ by a homotopy $(t+(1-t)d_k,(1-t))$. This homotopy is nondegenerate, since 
\[\det\left[\begin{smallmatrix}t+(1-t)d_k&1-t\\1-t&t+(1-t)d_k\end{smallmatrix}\right]=(t+(1-t)(d_k-1))((1-t)d_k+1)>0.\]
We have used that $d_k>1$ for $k=l+1,\ldots,n'$. 

The summands $\left(\left[\begin{smallmatrix}d_j&0\\0&d_{j+1}\end{smallmatrix}\right],I_2\right)$, where both $0<d_j,d_{j+1}<1$ can be first homotoped to $\left(0_2,I_2\right)$ by a simple linear nondegenerate homotopy $\left((1-t)\left[\begin{smallmatrix}d_{j}&0\\0&d_{j+1}\end{smallmatrix}\right],I_2\right)$. Let now $x(t)$ be some small smooth real function on $[0,1]$ that is compactly supported in $(0,1)$, with $x(1/2)\ne 0$. The homotopy $\left(\left[\begin{smallmatrix}t&e^{i\pi/4}x(t)\\-e^{i\pi/4}x(t)&t\end{smallmatrix}\right],(1-t)I_2\right),$ is nondegenerate since the determinant \[\det\left[\begin{smallmatrix}t&e^{i\pi/4}x&1-t&0\\-e^{i\pi/4}x&t&0&1-t\\1-t&0&t&e^{-i\pi/4}x\\0&1-t&-e^{-i\pi/4}x&t\end{smallmatrix}\right]=(2 t-1)^2+|x|^2(|x|^2+2(1-t)^2)>0.\] Joining all the homotopies, we get to the pair $(I_2,0_2)$.

Finally,  if $l$ is odd, the pair $(d_1,1)$ can be homotoped to $(0,1)$ by $((1-t),1)$.

Using the Remark \ref{split}, combining the homotopies of each of the blocks of $(D\oplus\Lambda,I_n)$, we get a nondegenerate homotopy to either $(I_n,0_n)$, if $l$ is even, or $(0,1)\oplus (I_{n-1},0_{n-1})$ if $l$ is odd. This completes the proof. 
\end{proof}

\begin{proof}[Proof of theorem \ref{thm1}] After a small perturbation, we can assume that $Y$ has only isolated elliptic or hyperbolic complex points, and we can also assume those points are quadratic, meaning that $o(|z|^2)=0$ in the form (\ref{form}) in some local coordinates. This is done by a small isotopy in the neighborhoods of complex points, canceling the $o(|z|^2)$ term. Let $p$ be any complex point on $Y$, represented locally by a pair $(A,B)\in\cM_n$. Let $(A_t,B_t)$ be a nondegenerate homotopy of pairs, constructed in Proposition \ref{connected}, connecting either $(I_n,0_n)$, if $p$ is elliptic, or $(0,1)\oplus (I_{n-1},0_{n-1})$, if $p$ is hyperbolic, to $(A,B)$. We can assume that the homotopy is constant near $t=0$ and $t=1$. Let $\varepsilon$ be small and let $\widetilde Y$ coincide with $Y$ away from $|z|<\varepsilon$, and is given by
$$w=\bar{z}^TA(\sqrt[n]{|z|/\varepsilon})z+\Real (z^TB(\sqrt[n]{|z|/\varepsilon})z),$$
for $|z|\le\varepsilon$. It is shown in the proof of Theorem 1 in \cite{S1} that $p=0$ is the only complex point of the surface $\tilde Y$ in a neighborhood of $|z|\le\epsilon$. By construction, the complex point $p=0$ is of the desired form. 
\end{proof}

To prove Theorem $2$, we need the following  two lemmas.

\begin{lemma}\label{lem1} Let $(z_1,\ldots,z_n,w)=(z,w)$ be coordinates on $\C^n\times \C$. Let $Y=\{(z,w)\in\C^n;w=\frac{1}{2}\bar z_1^2+\frac{1}{2}\bar z_2^2+\cdots+\frac{1}{2}\bar z_n^2\}.$   The function 
\[f(z,w)=(1+|z|^2)\left|w-\frac{1}{2}\bar z_1^2-\frac{1}{2}\bar z_2^2-\cdots-\frac{1}{2}\bar z_n^2\right|^2\]
is plurisubharmonic in a neighborhood $U$ of $0$ and strictly plurisubharmonic in $U\backslash Y$. Furthermore, the Levi form of $f$, restricted to $TY/T^\C Y$, is strictly positive. 
\end{lemma}
\begin{proof} Let us denote $\psi(z,w)=w-\frac{1}{2}\bar z_1^2-\frac{1}{2}\bar z_2^2-\cdots-\frac{1}{2}\bar z_n^2$. Then the Levi form of
$(1+|z|^2)\psi\bar{\psi}$ equals to

\begin{align*} &\frac{i}{2}\left((1+|z|^2)dw\wedge d\bar w +\psi\bar\psi\sum_{k=1}^n dz_k\wedge d\bar z_k 
+(1+|z|^2)(\sum_{k=1}^{n}z_idz_i)\wedge(\sum_{k=1}^n\bar{z_k}d\bar z_k)\right.\\ &\left.-\psi(\sum_{k=1}^nz_kdz_k)\wedge(\sum_{k=1}^{n}z_k d\bar{z_k})-\bar\psi(\sum_{k=1}^n\bar z_k dz_k)\wedge(\sum_{k=1}^n \bar z_k d\bar z_k ) 
+\bar\psi dw\wedge \sum_{k=1}^{n}z_kd\bar z_k\right.\\
&\left. -\psi d\bar{w}\wedge\sum_{k=1}^{n}\bar z_k dz_k\right).
\end{align*}
The Levi form, evaluated on a tangent vector $(Z,W)$, gives us
\begin{align*}&(1+|z|^2)|W|^2+\psi\bar{\psi}|Z|^2+(1+|z|^2)|zZ|^2-2\Real(\psi(zZ)(z\bar Z))\\&+2\Real(\psi\bar W(\bar z Z)).\end{align*}	
If $Z=0$, the form is clearly positive. So we can assume $|Z|=1$ and let $\alpha,\beta$, $0\le|\alpha|,|\beta|\le 1$ be such that $zZ=\alpha |z|$ and $z\bar Z=\beta |z|$. We get
\begin{align*}\nonumber&(1+|z|^2)|W|^2+|\psi|^2+(1+|z|^2)|z|^2|\alpha|^2-2\Real(\psi|z|^2\alpha\beta)+\\&2\Real(\psi\bar W|z|\bar\beta)\ge |W|^2+|\psi|^2+|z|^2|\alpha|^2-2|\psi||z|^2|\alpha||\beta|-2|\psi||W||z||\beta|\\&=(|W|-|\psi||z||\beta|)^2+|z|^2(|\alpha|-|\psi||\beta|)^2+|\psi|^2(1-2|\beta|^2|z|^2).\end{align*}	
As long as $|z|^2<\frac{1}{2}$ and $|\psi|\ne 0$, this is strictly positive. Restricted to $\psi=0$, the form only vanishes if $W=0$ and $zZ=0$, which is exactly along $T^\C Y$. 
\end{proof}

\begin{lemma}\label{lem2} Let $(z_1,z_2\ldots,z_n,w)=(z_1,z',w)$ be coordinates on $\C \times\C^{n-1}\times \C$. Let $Y=\{(z_1,z',w)\in\C^{n+1};w=|z_1|^2+\frac{1}{2}\bar z_2^2+\frac{1}{2}\bar z_3^2+\cdots+\frac{1}{2}\bar z_n^2\}.$ The function 
\[f(z,w)=(1+|z'|^2)\left|w-|z_1|^2-\frac{1}{2}\bar z_2^2-\frac{1}{2}\bar z_3^2-\cdots-\frac{1}{2}\bar z_n^2\right|^2\]
is $n$-pseudoconvex in $U$ and $n$-strictly pseudoconvex in $U\backslash Y$, where $U$ is some small neighborhood of $0$. Furthermore, the Levi form of $f$ is semipositive on $U\cap Y$ and  is strictly positive on $TY/T^\C Y$. 
\end{lemma}

\begin{proof} Let us write tangent vectors to $\C^{n+1}$ as $(Z_1,Z',W)$ corresponding to coordinates $(z_1,z',w)$ and let $\psi(z_1,z',w)=w-|z_1|^2-\frac{1}{2}\bar z_2^2-\frac{1}{2}\bar z_3^2-\cdots-\frac{1}{2}\bar z_n^2$. The Levi form of $f$ calculated on $(Z_1,Z',W)$ is
\begin{align*}&|\psi|^2|Z'|^2+(1+|z'|^2)|W|^2+2(1+|z'|^2)|z_1|^2|Z_1|^2+(1+|z'|^2)|z'Z'|^2\\
&-2(1+|z'|^2)\Real(z_1\bar Z_1 W)+2(1+|z'|^2)\Real((z_1\bar Z_1)(z'Z'))\\
&+2\Real(\psi(\bar z'Z')\bar W)-2\Real (\psi(z'Z')(z'\bar Z'))-2\Real(\psi(z_1\bar Z_1)(z'Z'))\\
&-2\Real(\psi(\bar z_1 Z_1)(z'\bar Z'))-2\Real\psi(1+|z'|^2)|Z_1|^2.
\end{align*}
An almost identical calculation, as in the proof of Lemma \ref{lem1}, shows that the above expression is strictly positive on nonzero vectors of the form $(0,Z',W)$, if $\phi\ne 0$. The form is negative on $(1,0,0)$ if $\phi\ne 0$. Along $\psi=0$, it is easily seen that the form only vanishes if both $\bar z_1 Z_1+z'Z'\ne 0$ and $W=\bar z_1 Z_1$ hold, which is exactly on $T^\C Y$. This completes the proof.   
\end{proof}

\begin{proof}[Proof of Theorem \ref{thm2}] Let $Y$ be a smoothly embedded real compact $n$-manifold in a complex $n+1$ dimensional manifold $X$. After a smooth $\cC^0$ small isotopy, we can assume that complex points of $Y$ are isolated and are either of the type $w=\frac{1}{2}(\bar z_1^2+\bar z_2^2+\cdots+\bar z_n^2)$, or $w=|z_1|^2+\frac{1}{2}(\bar z_2^2+\bar z_2^3+\cdots+\bar z_n^2)$ ( Theorem \ref{thm1}).  Let $p_1,\ldots,p_k$ be complex points of $Y$ and let $f_1,\ldots,f_n$ be functions constructed in Lemma \ref{lem1} or Lemma \ref{lem2}, defined on small disjunct neighborhoods $U_1,\ldots,U_k$ of complex points. By (\cite{Ch}, Proposition 6.5), there exists a $2$-strictly pseudoconvex nonnegative function $f$, defined in a neighborhood of $Y\backslash\{U_1\cup\cdots\cup U_k\}$ in $X$, vanishing exactly on $Y$, whose Levi form is nonnegative along $Y$ and strictly positive on $TY/T^\C Y$. 
By patching $f$ with functions $f_1,\ldots,f_n$ we get a nonnegative function $\phi$, defined in a neighborhood $V$ of $Y$ with the following properties:
\begin{itemize}[leftmargin=*]
\item $\phi\ge 0$ and $\{\phi=0\}=Y$,
\item $\phi$ is $2$-strictly pseudoconvex, except at complex points, 
\item there exist positive constants $c,\ C$ so that $c (\dist(q,Y))^2\le\phi(q)\le C (\dist(q,Y))^2$ and $c(\dist(q,Y))\le |\grad \phi(q)|\le C(\dist(q,Y))$.  
\end{itemize} 
where $\dist(q,Y)$ is the distance of $q$ to $Y$ in some smooth metric on $X$.
Sublevel sets of $\phi$ are $n$-strictly pseudoconvex and give a tubular neighborhood basis of $Y$.  
\end{proof}

\begin{remark} In the case of real surface in complex surfaces, the real surface is always locally polynomially convex at hyperbolic complex points \cite{FS}, while the local polynomial hull at elliptic complex points is $3$ dimensional and foliated by Bishop discs. In higher dimensions, the local polynomial hall at complex points of the form $w=|z|^2+\sum_{k=1}^{n}\frac{\gamma_k}{2}\left(z^2+\bar z^2\right)+o(|z|^2),$ $0\le\gamma_1,\ldots,\gamma_n<1$ is also $2n+1$ dimensional and foliated by Levi flat hypersurfaces, if the nearby points are all nonminimal \cite{DTZ1}. We have shown above, if $n$ is even, that such a complex point can be put in a form
$w=\bar z_1^2+\bar z_2^2+\cdots+\bar z_n^2$
by a small isotopy, and by Lemma \ref{lem1} the new surface is locally polynomially convex at the complex point.  
\end{remark}

\bibliographystyle{amsart}

\end{document}